\documentclass[11pt]{amsart}%
\usepackage{palatino, mathpazo}
\usepackage{amsfonts}
\usepackage{amsmath}
\usepackage{amssymb,latexsym}
\usepackage{graphicx}
\usepackage[mathscr]{eucal}
\usepackage{amssymb}%
\providecommand{\U}[1]{\protect \rule{.1in}{.1in}}

\newtheorem{theorem}{Theorem}[section]

\newtheorem{lemma}[theorem]{Lemma}

\newtheorem{proposition}[theorem]{Proposition}
\theoremstyle{remark}
\newtheorem{remark}[theorem]{Remark}

\numberwithin{equation}{section}
\begin{document}
\title[Darboux-Treibich-Verdier potential and spectrum]{A necessary and sufficient condition for the Darboux-Treibich-Verdier potential with its spectrum contained in $\mathbb{R}$}
\author{Zhijie Chen}
\address{Department of Mathematical Sciences, Yau Mathematical Sciences Center,
Tsinghua University, Beijing, 100084, China }
\email{zjchen2016@tsinghua.edu.cn}
\author{Erjuan Fu}
\address{Yau Mathematical Sciences Center,
Tsinghua University, Beijing, 100084, China}
\email{fuerjuan@gmail.com}
\author{Chang-Shou Lin}
\address{Center for Advanced Study in
Theoretical Sciences (CASTS), National Taiwan University, Taipei 10617, Taiwan }
\email{cslin@math.ntu.edu.tw}

\begin{abstract}
In this paper, we study the spectrum of the complex Hill operator $L=\frac{d^2}{dx^2}+q(x;\tau)$ in $L^2(\mathbb{R},\mathbb{C})$ with the Darboux-Treibich-Verdier potential
\[q(x;\tau):=-\sum_{k=0}^{3}n_{k}(n_{k}+1)\wp \left(
x+z_0+\tfrac{\omega_{k}}{2};\tau \right),\]
where $n_k\in\mathbb{Z}_{\geq 0}$ with $\max n_k\geq 1$ and $z_0\in\mathbb{C}$ is chosen such that $q(x;\tau)$ has no singularities on $\mathbb{R}$. For any fixed $\tau\in i\mathbb{R}_{>0}$, we give a necessary and sufficient condition on $(n_0,n_1,n_2,n_3)$ to guarantee that the spectrum $\sigma(L)$ is
 \[\sigma(L)=(-\infty, E_{2g}]\cup[E_{2g-1}, E_{2g-2}]\cup \cdots \cup[E_{1}, E_{0}],\quad E_j\in \mathbb{R},\]
 and hence generalizes Ince's remarkable result in 1940 for the Lam\'{e} potential to the Darboux-Treibich-Verdier potential. We also determine the number of (anti)periodic eigenvalues in each bounded interval $(E_{2j-1}$, $E_{2j-2})$, which generalizes the recent result in \cite{HHV} where the Lam\'{e} case $n_1=n_2=n_3=0$ was studied.
\end{abstract}
\maketitle

\section{Introduction}

Let $\tau \in \mathbb{H}=\{  \tau|\operatorname{Im}\tau>0\}$ and $E_{\tau}:=\mathbb{C}/(\mathbb{Z}+\mathbb{Z}\tau)$ be a flat torus.
Recall that $\wp(z)=\wp(z;\tau)$ is the
Weierstrass elliptic function with basic periods $\omega_{1}=1$ and $\omega_{2}%
=\tau$. Denote also $\omega_{0}=0$ and $\omega_{3}=1+\tau$.
In this paper, we study the \emph{Darboux-Treibich-Verdier potential} (DTV potential for short) \cite{Darboux,TV,Veselov}: \begin{equation}\label{Tre}
q^{\mathbf{n}}(z;\tau):=-\sum_{k=0}^{3}n_{k}(n_{k}+1)\wp \left(
z+\tfrac{\omega_{k}}{2};\tau \right),
\end{equation}
where $\mathbf{n}=(n_{0},n_{1},n_{2},n_{3})$ and $n_k\in\mathbb{Z}_{\geq 0}$ for all $k$ with $\max_k n_k\geq 1$. Without loss of generality, we always assume
\[n_0=\max_k n_k\geq 1.\]
If $n_{1}=n_{2}=n_{3}=0$, then $q^{\mathbf{n}}(z;\tau)$ becomes the classical Lam\'{e} potential \cite{Lame} \[q_n(z;\tau):=-n(n+1)\wp(z;\tau).\]

The DTV potential $q^{\mathbf{n}}(z;\tau)$ is famous
as an algebro-geometric finite-gap potential associated with the stationary
KdV hierarchy. We refer the readers to \cite{CLW,CKL1,CKL2,GW1,Tak1,Tak2,Tak3,Tak4,Tak5,TV,Veselov}
and references therein for historical reviews and subsequent developments. In the literature, a potential $q(z)$ is called an \emph{algebro-geometric finite-gap
potential }if there is an odd-order differential operator
\begin{equation}\label{odd-op}
P_{2g+1}=\left( \frac{d}{dz} \right)^{2g+1}%
+\sum_{j=0}^{2g-1}b_{j}(z)\left( \frac{d}{dz} \right)^{2g-1-j}
\end{equation} such that $[P_{2g+1}, d^{2}/dz^{2}
+q(z)]=0$, that is, $q(z)$ is a solution of stationary KdV hierarchy equations (cf.
\cite{GH-Book,GW}).

For the DTV potential $q^{\mathbf{n}}(z;\tau)$, we let $P_{2g+1}$ be the unique operator of the form (\ref{odd-op}) satisfying $[P_{2g+1}, d^{2}/dz^{2}
+q^{\mathbf{n}}(z;\tau)]=0$ such that its order $2g+1$ is \emph{smallest}. Then
a celebrated theorem of Burchnall and Chaundy \cite{Burchnall-Chaundy} implies the existence of the so-called {\it spectral polynomial} $Q^{\mathbf{n}}(E;\tau)$ of degree $2g+1$ in $E$ associated to $q^{\mathbf{n}}(z;\tau)$ such that
\begin{equation}\label{eqeq}P_{2g+1}^2=Q^{\mathbf{n}}(\tfrac{d^{2}}{dz^{2}}
+q^{\mathbf{n}}(z;\tau);\tau).\end{equation}
The number
$g$, i.e. the arithmetic genus of the associate hyperelliptic curve $\{(E,W)|W^{2}=Q^{\mathbf{n}}(E;\tau)\}$, was
computed in \cite{GW1, Tak5}: Let $m_{k}$ be the rearrangement of $n_{k}$ such
that $m_{0}\geq m_{1}\geq m_{2}\geq m_{3}$, then
\begin{equation}
g=%
\begin{cases}
m_{0} & \text{if $\sum m_{k}$ is even and $m_{0}+m_{3}\geq m_{1}+m_{2}$};\\
\frac{m_{0}+m_{1}+m_{2}-m_{3}}{2} & \text{if $\sum m_{k}$ is even and
$m_{0}+m_{3}<m_{1}+m_{2}$};\\
m_{0} & \text{if $\sum m_{k}$ is odd and $m_{0}>m_{1}+m_{2}+m_{3}$};\\
\frac{m_{0}+m_{1}+m_{2}+m_{3}+1}{2} & \text{if $\sum m_{k}$ is odd and
$m_{0}\leq m_{1}+m_{2}+m_{3}$}.
\end{cases}
\label{genus}%
\end{equation}
Furthermore, it is known (cf. \cite{GW1,Tak1,Tak5}) that the roots of $Q^{\mathbf{n}}(\cdot
;\tau)=0$ are \textit{distinct} for generic $\tau \in \mathbb{H}$ and \begin{equation}\label{eq-22}Q^{\mathbf{n}}(E;\tau)\in \mathbb{R}[E]\quad\text{ for }\quad\tau\in i\mathbb{R}_{>0}.\end{equation}

The spectral polynomial plays an important role in the spectral theory of the associated Hill operator. Since the DTV potential is doubly periodic, there are two such operators related to the two periods respectively. In this paper, we study the operator along the $\omega_1=1$ direction, i.e. we study the spectrum $\sigma(L)$ of the following Hill operator
\begin{equation}\label{hill-operator}
L=\frac{d^2}{dx^2}+q(x;\tau),\quad  x\in\mathbb{R},\;\;\text{with} \;\; q(x;\tau)=q^{\mathbf{n}}(x+z_0;\tau)\end{equation}
in $L^2(\mathbb{R},\mathbb{C})$,
where $z_0\in\mathbb{C}$ is chosen such that $q(x;\tau)$ has no singularities on $\mathbb{R}$. The spectral theory of the complex Hill operator has attracted significant attention and has been studied widely in the literature; see e.g. \cite{BG-JAM,B-CPAM,GW,GW2,HHV,Rofe-Beketov} and references therein.

Suppose for some $\tau\in i\mathbb{R}_{>0}$ that all roots
of the spectral polynomial $Q^{\mathbf{n}}(\cdot
;\tau)$ are real and distinct, denoted by $E_{2g}<E_{2g-1}<\cdots<E_{1}<E_{0}$, then we proved in \cite[Lemma 3.6]{CL-AJM} (we will recall it in Lemma \ref{lemma5}) that the spectrum $\sigma(L)$ is
\begin{equation}
\label{spectrum0-1}\sigma(L)=(-\infty, E_{2g}]\cup[E_{2g-1}, E_{2g-2}]\cup \cdots \cup[E_{1}, E_{0}].
\end{equation}
This result was first discovered by Ince in the seminal work \cite{Ince}, where he proved that (\ref{spectrum0-1}) holds for the Lam\'{e} case $L=\frac{d^2}{dx^2}-n(n+1)\wp(x+\frac{\omega_k}{2};\tau)$ with $k\in\{2,3\}$. His proof essentially relies on the fact that $\wp(x+\frac{\omega_k}{2};\tau)$ with $k\in\{2,3\}$ is \emph{real-valued} and smooth on $\mathbb{R}$, and hence does not work for the general DTV case.

\subsection{Real and distinct roots}
In this paper, we study two problems related to the spectrum $\sigma(L)$ of the operator $L$ in (\ref{hill-operator}). The first one is
whether the spectrum $\sigma(L)$ for the DTV potential is of the form (\ref{spectrum0-1}) or not, or equivalently,

\medskip

\noindent \textbf{(Q$_1$)}: {\it Whether are all roots of the spectral polynomial $Q^{\mathbf{n}}(\cdot;\tau)$ with $\tau\in i\mathbb{R}_{>0}$ real and distinct?}

\medskip

For the Lam\'{e} case, the answer for \textbf{(Q$_1$)} is Yes as mentioned before. However, it is not necessarily true for all the DTV potentials; see e.g. \cite[Remark 4.2]{CL-AJM} for a counterexample.
Thus further assumptions on $n_k$'s are needed. See \cite{CKLT,GW1,Tak1} for some sufficient (but not necessary) conditions on $n_k$'s. Here we introduce two relations:
\begin{equation}
\label{c1}\frac{n_{1}+n_{2}-n_{0}-n_{3}}{2}\geq1,\quad n_{1}\geq1,\quad
n_{2}\geq1,
\end{equation}
\begin{equation}
\label{c2}\frac{n_{0}+n_{3}-n_{1}-n_{2}}{2}\geq 1,\quad n_{0}\geq1,\quad
n_{3}\geq1.
\end{equation}
Recently, we obtained an almost complete answer to \textbf{(Q$_1$)} in \cite{CL-AJM}.

\medskip

\noindent {\bf Theorem A.} \cite{CL-AJM}
\emph{All the roots of $Q^{\mathbf{n}}(\cdot;\tau)$ are real
and distinct for {\bf every} $\tau \in i\mathbb{R}_{>0}$ if and only if $\mathbf{n}$ satisfies neither (\ref{c1}) nor (\ref{c2}).}
\medskip

To emphasize the importance of Theorem A, we mention one application to the following mean field equation
\begin{equation}\label{mfe}
\Delta u+e^u=8\pi \sum_{k=0}^3 n_k\delta_{\frac{\omega_k}{2}},\quad \text{on }\;E_{\tau},
\end{equation}
where $\delta_{\frac{\omega_k}{2}}$ is the Dirac measure at $\frac{\omega_k}{2}$.
Theorem A is the crucial step to prove the following non-existence result.

\medskip

\noindent {\bf Theorem B.} \cite{CL-AJM}
\emph{Equation (\ref{mfe}) has no even solutions for all $\tau \in i\mathbb{R}_{>0}$ if and only if $\mathbf{n}$ satisfies neither (\ref{c1}) nor (\ref{c2}).}
\medskip

In this paper, we succeed to delete the condition "every" in Theorem A via a new observation, and hence give the complete answer to \textbf{(Q$_1$)}.
Our first result is

\begin{theorem}\label{sharp-realroot} The following statements are equivalent.
\begin{itemize}
\item[(1)] $Q^{\mathbf{n}}(\cdot;\tau)$ has at least two roots in $\mathbb{C}\setminus\mathbb{R}$ for some $\tau \in i\mathbb{R}_{>0}$;
\item[(2)] $Q^{\mathbf{n}}(\cdot;\tau)$ has at least two roots in $\mathbb{C}\setminus\mathbb{R}$ for all $\tau \in i\mathbb{R}_{>0}$;
\item[(3)] $\mathbf{n}$ satisfies either (\ref{c1}) or (\ref{c2}).
\end{itemize}
In this case, the spectrum $\sigma(L)\not\subset\mathbb{R}$ is still symmetric with respect to $\mathbb{R}$ but not of the form (\ref{spectrum0-1}).
\end{theorem}

We end this subsection by proposing some open problems. Let $D(\tau)$ denote the discriminant of $Q^{\mathbf{n}}(E;\tau)$ as a polynomial in $E$. It is easy to see that $D(\tau)$ is a modular form with respect to $\Gamma(2)=\{A\in SL(2,\mathbb{Z}) \,|\, A\equiv I_2\,\operatorname{mod}\, 2\}$ and vanishes at $\infty$. By the aforementioned result proved by Takemura \cite{Tak5} that $Q^{\mathbf{n}}(\cdot;\tau)$ has distinct roots expect for a discrete set of $\tau$'s, we have $D(\tau)\not\equiv 0$. We propose
\medskip

\noindent{\bf Conjecture.} \emph{$D(\tau)$ has at most simple zeros in $\mathbb{H}$, namely if $D(\tau)=0$ then $D'(\tau)\neq 0$}.
\medskip

It is also interesting to ask whether $D(\tau)$ vanishes at other cusps and to compute its vanishing order.
For the case $\mathbf{n}$ satisfying either (\ref{c1}) or (\ref{c2}), since $Q^{\mathbf{n}}(E;\tau)$ with $\tau=ib$, $b>0$, always has complex roots, another open problem is to determine the number of pairs of complex roots (for large $b$).

\subsection{Location of (anti)periodic eigenvalues}
The second problem is to study (anti)periodic eigenvalues of $L$.
Recall that $E\in\mathbb{C}$ is called a periodic (resp. antiperiodic) eigenvalue of $L$ if $Ly=Ey$ has a nonzero solution $y$ satisfying $y(x+1)=y(x)$ (resp. $y(x+1)=-y(x)$).
It is well known (cf. \cite{GW}) that the operator $L$ in (\ref{hill-operator}) has \emph{countably many} periodic and antiperiodic eigenvalues, which contain all roots of the spectral polynomial $Q^{\mathbf{n}}(\cdot;\tau)$ as a proper subset. Denote
\begin{equation}
\sigma_{p}(L):=\{E\,|\, \text{$E$ is a (anti)periodic eigenvalue of $L$, $Q^{\mathbf{n}}(E;\tau)\neq 0$}\}.
\end{equation}
Clearly $\sigma_p(L)\subset \sigma(L)$.
Concerning the positions of those $E\in\sigma_{p}(L)$, Haese-Hill et al. \cite{HHV} proved that

\medskip

\noindent {\bf Theorem C.} \cite{HHV}
\emph{
For the Lam\'{e} case $n_1=n_2=n_3=0$ with $\tau\in i\mathbb{R}_{>0}$, there holds
\begin{equation}\label{eq-open-gap}
\sigma_{p}(L)\cap (E_{2j-1}, E_{2j-2})=\emptyset\;\; \forall\, 1\leq j\leq n, \; \text{i.e.}\; \sigma_{p}(L)\subset (-\infty, E_{2n}).
\end{equation}}

Let $\Delta(E;\tau)$ be the Hill's discriminant of the operator $L$ in (\ref{hill-operator}), then $E$ is a periodic (resp. antiperiodic) eigenvalue if and only if $\Delta(E;\tau)=2$ (resp. $\Delta(E;\tau)=-2$); see Section \ref{Floquet-theory} for a brief overview of this entire function $\Delta(E;\tau)$. Theorem C indicates that for the Lam\'{e} case,
\[\Delta(E_{2j-1};\tau)\Delta(E_{2j-2};\tau)=-4\;\; \forall\, 1\leq j\leq n.\]
This sign information is also important because it is invariant if we consider the deformation of $\tau$.

We want to generalize Theorem C to the DTV potentials. Assume that $\mathbf{n}$ violates both (\ref{c1}) and (\ref{c2}), then Theorem A says that the spectrum $\sigma(L)$ is given by (\ref{spectrum0-1}), and it is easy to see that one of the following hold (note $n_0=\max_k n_k\geq 1$)
\begin{enumerate}
\item[(a)] \emph{either $n_0\geq n_1+n_2+1$ with $n_3=0$ or $n_0+n_3=n_1+n_2$};
\item[(b)] $n_0+n_3=n_1+n_2-1$;
\item[(c)] $n_0+n_3=n_1+n_2+1$ \emph{with} $n_3\geq 1$.
\end{enumerate}
Recalling (\ref{genus}), we obtain
\begin{equation}
g=%
\begin{cases}
n_{0} & \text{in Case (a)};\\
n_0+n_3+1 &\text{in Case (b)};\\
n_0+n_3 & \text{in Case (c)}.
\end{cases}
\label{genus1}%
\end{equation}
Define a new integer
\begin{equation}
m:=%
\begin{cases}
n_{0}-n_1 & \text{in Case (a)};\\
n_2+n_3+1 &\text{in Case (b)};\\
n_2+n_3+1 & \text{in Case (c) with $n_0>n_2$},\\
n_2+n_3   &\text{in Case (c) with $n_0=n_2$}.
\end{cases}
\label{genus2}%
\end{equation}
Clearly $g\geq m$.
Then our next result shows that (\ref{eq-open-gap}) does not necessarily hold for all the DTV potentials.

\begin{theorem}\label{thm-TV-open-gap}
Let $\mathbf{n}$ satisfy neither (\ref{c1}) nor (\ref{c2}), with $(g,m)$ given in (\ref{genus1})-(\ref{genus2}) and $\tau\in i\mathbb{R}_{>0}$. Then for the operator $L$ in (\ref{hill-operator}), there holds
\begin{align*}
&\sigma_{p}(L)\cap (E_{2j-1}, E_{2j-2})=\emptyset\quad \forall 1\leq j\leq m,\\
&\sigma_{p}(L)\cap (E_{2j-1}, E_{2j-2})=\text{one point}\quad \forall m+1\leq j\leq g.
\end{align*}
In particular, $\Delta(E_{2m-1};\tau)=\Delta(E_{2m};\tau)=\cdots=\Delta(E_{2g};\tau)=(-1)^m 2$.
\end{theorem}

For the case $\mathbf{n}$ satisfying either (\ref{c1}) or (\ref{c2}), the spectrum $\sigma(L)$ is not of the form (\ref{spectrum0-1}), but it is still very interesting to study the location of (anti)periodic eigenvalues. We expect that the results should be much more complicated than Theorem \ref{thm-TV-open-gap}.

This paper is organized as follows. In Section 2, we briefly review the spectral theory of Hill equation from \cite{GW} and apply it to the DTV potentials. In Section \ref{general-lame}, we develop further our ideas in \cite{CL-AJM} to prove Theorem \ref{sharp-realroot}, where we prefer to provide all the necessary details to make the paper self-contained. Theorem \ref{thm-TV-open-gap} will be proved in Section \ref{sec-periodic}, where we will apply some results from \cite{Tak5}.

\section{Spectral theory \cite{GW}}

\label{Floquet-theory}

In this section, we briefly review the spectral theory of Hill
equation with \emph{complex-valued} potentials from \cite{GW} and apply it to the DTV potential; see Theorem 2.A, which will be used frequently in the proofs of Theorems \ref{sharp-realroot}-\ref{thm-TV-open-gap} in Sections \ref{general-lame}-\ref{sec-periodic}.

Let $q(x)$ is a complex-valued continuous nonconstant periodic function of
period $\Omega$ on $\mathbb{R}$. Consider the following Hill equation%
\begin{equation}
y^{\prime \prime}(x)+q(x)y(x)=Ey(x),\text{ \  \ }x\in \mathbb{R}. \label{eq2-1}%
\end{equation}
This equation has received an enormous amount of consideration due to its
ubiquity in applications as well as its structural richness; see e.g.
\cite{GW,GW2} and references therein for historical reviews.

Let $y_{1}(x)$ and $y_{2}(x)$ be any two linearly independent solutions of
(\ref{eq2-1}). Then so do $y_{1}(x+\Omega)$ and $y_{2}(x+\Omega)$ and
hence there is a monodromy matrix $M(E)\in SL(2,\mathbb{C})$ such that
\[
(y_{1}(x+\Omega),y_{2}(x+\Omega))=(y_{1}(x), y_{2}(x))M(E).
\]
Define the \emph{Hill's discriminant} $\Delta (E)$ by
\begin{equation}
\label{trace}\Delta(E):=\text{tr}M(E),
\end{equation}
which is clearly an invariant of (\ref{eq2-1}), i.e. does not depend on the choice of linearly
independent solutions. This entire function $\Delta(E)$ encodes all information of the spectrum
$\sigma(L)$ of the operator $L=\frac{d^{2}%
}{dx^{2}}+q(x)$; see e.g. \cite{GW2} and references
therein. Indeed, we define
\begin{equation}
\mathcal{S}:=\Delta^{-1}([-2,2])=\{E\in \mathbb{C}\,|\, -2\leq \Delta(E)\leq2\}
\end{equation}
to be the \textit{conditional stability set} of the operator $L=\frac{d^{2}%
}{dx^{2}}+q(x)$.
Then Rofe-Beketov \cite{Rofe-Beketov} proved that $S$ coincides with the spectrum:
\begin{equation}\label{spe-S}\sigma(L)=\mathcal{S}=\{E\in \mathbb{C}\,|\, -2\leq \Delta(E)\leq2\}.\end{equation}
This important fact will play a key role in this paper.

Clearly $E$ is a (anti)periodic eigenvalue if and only if $\Delta(E)=\pm 2$. Define
\[
d(E):=\text{ord}_{E}(\Delta(\cdot)^{2}-4).
\]
Then it is well known (cf. \cite[Section 2.3]{Naimark}) that $d(E)$ equals to \textit{the algebraic multiplicity of (anti)periodic
eigenvalues}. Let $c(E,x,x_{0})$ and $s(E,x,x_{0})$ be the special fundamental
system of solutions of (\ref{eq2-1}) satisfying the initial values
\[
c(E,x_{0},x_{0})=s^{\prime}(E,x_{0},x_{0})=1,\ c^{\prime}(E,x_{0}%
,x_{0})=s(E,x_{0},x_{0})=0.
\]
Then we have
\[
\Delta(E)=c(E,x_{0}+\Omega,x_{0})+s^{\prime}(E,x_{0}+\Omega,x_{0}).
\]
Define
\[
p(E,x_{0}):=\text{ord}_{E}s(\cdot,x_{0}+\Omega,x_{0}),
\]%
\[
p_{i}(E):=\min \{p(E,x_{0}):x_{0}\in \mathbb{R}\}.
\]
It is known  that $p(E,x_{0})$ is the algebraic
multiplicity of a Dirichlet eigenvalue on the interval $[x_0, x_0+\Omega]$,
and $p_{i}(E)$ denotes the immovable part of $p(E,x_{0})$ (cf. \cite{GW}). It
was proved in \cite[Theorem 3.2]{GW} that $d(E)-2p_{i}(E)\geq0$. Define
\begin{equation}
D(E):=E^{p_{i}(0)}\prod \limits_{\lambda \in \mathbb{C}\backslash \{0\}} \left(
1-\tfrac{E}{\lambda}\right)  ^{p_{i}(\lambda)}. \label{eq3-2}%
\end{equation}

Now we consider the operator $L$ in (\ref{hill-operator}), i.e. $q(x)=q(x;\tau)=q^{\mathbf{n}}(x+z_0;\tau)$ is the DTV potential, which is smooth on $\mathbb{R}$ with period $\Omega=1$.
Applying the general result \cite[Theorem 4.1]{GW} to the DTV potential, we obtain

\medskip \noindent \textbf{Theorem 2.A.} \cite[Theorem 4.1]{GW} \textit{For the DTV potential $q(x)=q^{\mathbf{n}}(x+z_0;\tau)$, the following hold.}

\textit{(i) $d(E)-2p_{i}(E)>0$ on a finite set $\{E_{j}\}_{j=1}^{m}$ for some $m\in\mathbb{N}$ and $d(E)-2p_{i}(E)=0$
elsewhere, and the associated spectral polynomial $Q^{\mathbf{n}}(E;\tau)$ satisfies
\begin{equation}
\label{hyper}Q^{\mathbf{n}}(E;\tau)=\prod_{j=1}^{m} (E-E_{j})^{d(E_{j}%
)-2p_{i}(E_{j})}=C\frac{\Delta(E)^{2}-4}{D(E)^{2}}.
\end{equation}
Here $D(E)$ is seen in (\ref{eq3-2}) and $C$ is some nonzero constant. In particular, $2g+1=\deg Q^{\mathbf{n}}(E;\tau)=\sum_{j=1}^{m} (d(E_{j})-2p_{i}(E_{j}))$.}

\textit{(ii) the spectrum $\sigma(L)=\mathcal{S}$ consists of finitely many
bounded simple analytic arcs $\sigma_{k}$, $1\leq k\leq \tilde{g}$ for some $\tilde
{g}\leq g$ and one semi-infinite simple analytic arc $\sigma_{\infty}$ which tends to
$-\infty+\langle q\rangle$, with $\langle q\rangle=\int
_{x_{0}}^{x_{0}+1}q(x)dx$, i.e.
\[
\sigma(L)=\mathcal{S}=\sigma_{\infty}\cup \cup_{k=1}^{\tilde{g}}\sigma_{k}.
\]
Furthermore, the finite end points of such arcs must be those $E\in
\{E_{j}\}_{j=1}^{m}$ with $d(E)=2p_i(E)+\text{ord}_{E}Q^{\mathbf{n}}(\cdot;\tau)$ odd, and there are exactly $d(E)$'s semi-arcs of $\sigma(L)$ meeting at such $E$. } \medskip

\section{Proof of Theorem \ref{sharp-realroot}}

\label{general-lame}

The purpose of this section is to prove Theorem \ref{sharp-realroot}.
First we briefly explain why the spectrum $\sigma(L)$ does not depend on the choice of $z_0$. Consider the generalized
Lam\'{e} equation (GLE)
\begin{equation}
y^{\prime \prime}(z)=\bigg[  \sum_{k=0}^{3}n_{k}(n_{k}+1)\wp \left(
z+\tfrac{\omega_{k}}{2};\tau \right)  +E\bigg]  y(z),\quad z\in\mathbb{C}.\label{GLE-1}%
\end{equation}
It is known (cf. \cite{GW1,Tak1}) that the monodromy representation of GLE
(\ref{GLE-1}) is a group homomorphism $\rho(\cdot;E):\pi_{1}(E_{\tau
})\rightarrow SL(2,\mathbb{C})$ and so abelian. Let $\ell_{j}\in \pi_{1}(E_{\tau})$, $j=1,2$,
be the two fundamental cycles $z\to z+\omega_j$ and let $\rho(\ell_{j};E)$ denote the monodromy matrix of GLE (\ref{GLE-1})
with respect to any linearly independent solutions $(y_1,y_2)$, i.e.
\[
(y_{1}(z+\omega_j),y_{2}(z+\omega_j))=(y_{1}(z), y_{2}(z))\rho(\ell_{j};E),\quad j=1,2.
\]
Define
\begin{equation}
\label{tildes}\tilde{\mathcal{S}}^{\mathbf{n}}:=\{E\in \mathbb{C}\,|\, -2\leq \operatorname{tr} \rho
(\ell_{1}; E)\leq2\}.\end{equation}
Clearly $\operatorname{tr} \rho(\ell_{1}; E)$ and so $\tilde{\mathcal{S}}^{\mathbf{n}}$ are independent of the choice
of $(y_1,y_2)$.

\begin{lemma}\label{lem-spect-in}
The spectrum $\sigma(L)$ of the operator $L$ in (\ref{hill-operator}) satisfies
$\sigma(L)=\tilde{\mathcal{S}}^{\mathbf{n}}$, i.e. $\sigma(L)$ is independent of the choice of $z_0$.
\end{lemma}

\begin{proof}
Clearly if $(y_{1}(z),y_{2}(z))$ is a pair of linearly
independent solutions of GLE (\ref{GLE-1}), then $(w_{1}(x), w_{2}%
(x)):=(y_{1}(x+z_0),y_{2}(x+z_0))$ with $x\in \mathbb{R}$ is a pair of linearly
independent solutions of $Lw=Ew$. Thus, $\rho(\ell_{1};E)$ is also the monodromy
matrix $M(E)$ (defined in Section \ref{Floquet-theory}) of $Lw=Ew$, which gives
$\Delta(E)=\operatorname{tr}\rho(\ell_{1};E)$
and so we obtain the desired identity
$\sigma(L)=\tilde{\mathcal{S}}^{\mathbf{n}}$ by using (\ref{spe-S}) and (\ref{tildes}).
\end{proof}

\begin{lemma}
\label{lemma5} \cite[Lemma 3.6]{CL-AJM} Let $\tau \in i\mathbb{R}_{>0}$ and
suppose $Q^{\mathbf{n}}(E;\tau)$ has $2g+1$ real distinct zeros, denoted by
$E_{2g}<E_{2g-1}<\cdots<E_{1}<E_{0}$. Then the spectrum $\sigma(L)$ of the operator $L$ in (\ref{hill-operator}) satisfies
\begin{equation}
\label{spectrum1}\sigma(L)=\tilde{\mathcal{S}}^{\mathbf{n}}=(-\infty, E_{2g}]\cup[E_{2g-1}, E_{2g-2}]\cup \cdots \cup[E_{1}, E_{0}].
\end{equation}

\end{lemma}

\begin{proof} We sketch the proof here for later usage. Though the DTV potential $q^{\mathbf{n}}(z;\tau)$ is real-valued for $z\in\mathbb{R}$, it has poles at $\mathbb{Z}$ and $\frac{1}{2}+\mathbb{Z}$.
Instead, $q(x;\tau)=q^{\mathbf{n}}(x+z_0;\tau)$ is smooth on
$\mathbb{R}$ with period $\Omega=1$ but {\it not necessarily real-valued}. Thus the classic theory can not be applicable to this $q(x;\tau)$ to obtain (\ref{spectrum1}) either.

Under our assumptions, by (\ref{hyper}) in Theorem
2.A-(i)  we have
\begin{equation}\label{d-odd}
d(E_{j}):=\operatorname{ord}_{E_{j}}(\Delta(\cdot)^{2}-4)=1+2p_{i}(E_{j})\;\;\text{is \textit{odd} for all $j\in \lbrack0,2g]$.}
\end{equation}
On the other hand,
Theorem 2.A-(ii) says that: The spectrum $\sigma(L)$ consists of
finitely many bounded spectral arcs $\sigma_{k}$, $1\leq k\leq \tilde{g}$ for
some $\tilde{g}\leq g$ and {\it one} semi-infinite arc $\sigma_{\infty}$ which tends
to $-\infty+\langle q\rangle$, i.e.
\[
\sigma(L)=\sigma_{\infty}\cup \cup_{k=1}^{\tilde{g}}\sigma_{k}.
\]
Furthermore, the set of the finite end points of such arcs is precisely $\{E_{j}
\}_{j=0}^{2g}$ because of (\ref{d-odd}), and there are exactly $d(E_{j})$
semi-arcs of $\sigma(L)$ meeting at each $E_{j}$. Together these with the following three facts:

\begin{itemize}
\item[(a)] We proved in \cite[Lemma 3.5]{CL-AJM} that  $\tilde{\mathcal{S}}^{\mathbf{n}}$ is symmetric with respect to the
real line $\mathbb{R}$ if $\tau\in i\mathbb{R}_{>0}$, so does $\sigma(L)=\tilde{\mathcal{S}}^{\mathbf{n}}$ by Lemma \ref{lem-spect-in};

\item[(b)] A classical result (see e.g. \cite[Theorem 2.2]{GW2}) says that
$\mathbb{C}\setminus \sigma(L)$ is path-connected;

\item[(c)] Our assumption gives $E_{j}\in \mathbb{R}$ and $E_{2g}<E_{2g-1}%
<\cdots<E_{1}<E_{0}$;
\end{itemize}

\noindent we easily conclude that (i) $\sigma(L)\subset \mathbb{R}$, (ii) $d(E_{j})=1$
 for all $j$ and so (\ref{spectrum1}) holds.
Indeed, since (c) says that all finite end points of spectral arcs are on $\mathbb{R}$, the assertion (i) $\sigma(L)\subset \mathbb{R}$ follows immediately from (a)-(b). Consequently, there are at most two semi-arcs of $\sigma(L)$ meeting at each $E_j$. This, together with (\ref{d-odd}), yields the assertion (ii) $d(E_j)=1$ for all $j$, namely there is exactly one semi-arc of $\sigma(L)$ ending at $E_j$, which finally implies (\ref{spectrum1}).
\end{proof}

\begin{lemma}
\label{lemma3} \cite[Lemma 3.7]{CL-AJM} If
\[
Q^{(n_{0},n_{1},n_{2},n_{3})}(E;\tau)=\prod_{j=0}^{2g}(E-E_{j}(\tau)),
\]
Then
\begin{equation}
\label{modular1}Q^{(n_{0},n_{2},n_{1},n_{3})}(E;\tfrac{-1}{\tau})=\prod
_{j=0}^{2g}(E-\tau^{2}E_{j}(\tau)).
\end{equation}

\end{lemma}

The following result, which is our new observation comparing to \cite{CL-AJM}, is quite surprising to us. It plays a crucial role in proving Theorem \ref{sharp-realroot}.

\begin{proposition}\label{real-imply-distinct}
Let $\tau \in i\mathbb{R}_{>0}$ and suppose all zeros of $Q^{\mathbf{n}}(E;\tau)$ are real, denoted by
$E_{2g}\leq E_{2g-1}\leq \cdots\leq E_{1}\leq E_{0}$.
Then all the zeros are distinct, i.e. $E_i\neq E_j$ for $i\neq j$.
\end{proposition}

\begin{proof} In the following proof, we write $\tilde{\mathcal{S}}^{\mathbf{n}}=\tilde{\mathcal{S}}^{\mathbf{n}}
(\tau)$ to emphasize its dependence on $\tau$.
Note that
\[Q^{(n_{0}, n_{1},
n_{2}, n_{3})}(E;\tau)=\prod_{j=0}^{2g}(E-E_j),\quad E_j\in\mathbb{R}.\]
Then by the same proof as Lemma \ref{lemma5}, we have
\begin{equation}\label{S-tilde}\tilde{\mathcal{S}}^{(n_{0},n_{1},n_{2},n_{3})}
(\tau)=\sigma(L)=\sigma_{\infty}\cup \cup_{k=1}^{\tilde{g}}\sigma_{k}\subset\mathbb{R},\end{equation}
where $\tilde{g}\leq g$, $\sigma_{\infty}$ is the only semi-infinite arc which tends
to $-\infty$, and the set of the finite end points of such arcs is precisely those $\{E_{j}
\}_{0\le j\le 2g}$ with
\[d(E_j)=ord_{E_j}Q^{(n_{0}, n_{1},
n_{2}, n_{3})}(\cdot;\tau)+2p_i(E_j)\;\text{being odd}.\]
Since there are $d(E_{j})$
semi-arcs of $\tilde{\mathcal{S}}^{(n_{0},n_{1},n_{2},n_{3})}
(\tau)$ meeting at $E_{j}$, it follows from (\ref{S-tilde}) that $d(E_j)\le 2$, i.e. $p_i(E_j)=0$ and
\[ord_{E_j}Q^{(n_{0}, n_{1},
n_{2}, n_{3})}(\cdot;\tau)=d(E_j)\leq 2\quad\text{for all}\;j.\]
Furthermore,
\begin{equation}\label{Ejdouble}ord_{E_j}Q^{(n_{0}, n_{1},
n_{2}, n_{3})}(\cdot;\tau)=2\end{equation}
\[\Longleftrightarrow\text{$E_{j}$ is an interior point of $\tilde{\mathcal{S}}^{(n_{0},n_{1},n_{2},n_{3})}
(\tau)$},\]
and so
\begin{equation}\label{Sboundary}
\partial\tilde{\mathcal{S}}^{(n_{0},n_{1},n_{2},n_{3})}
(\tau)=\{-\infty\}\cup\{E_j\,|\, ord_{E_j}Q^{(n_{0}, n_{1},
n_{2}, n_{3})}(\cdot;\tau)=1\}.
\end{equation}

On the other hand, Lemma \ref{lemma3} and $\tau\in i\mathbb{R}_{>0}$ give
\[Q^{(n_{0},n_{2},n_{1},n_{3})}(E;\tfrac{-1}{\tau})=\prod
_{j=0}^{2g}(E-\tau^{2}E_{j}(\tau))
\]
with
\[\tau^2 E_0\leq \tau^2 E_1\leq\cdots\leq \tau^2 E_{2g-1}\leq \tau^2E_{2g}.\]
Therefore, the same argument as above shows that
\begin{equation}\label{Ejdouble1}ord_{E_j}Q^{(n_{0}, n_{1},
n_{2}, n_{3})}(\cdot;\tau)
=ord_{\tau^2E_j}Q^{(n_{0},n_{2},n_{1},n_{3})}(\cdot;\tfrac{-1}{\tau})=2\end{equation}
\[\Longleftrightarrow\text{$\tau^2E_{j}$ is an interior point of $\tilde{\mathcal{S}}^{(n_{0},n_{2},n_{1},n_{3})}
(\tfrac{-1}{\tau})\subset\mathbb{R}$},\]
and so
\begin{equation}\label{Sboundary1}
\partial\tilde{\mathcal{S}}^{(n_{0},n_{2},n_{1},n_{3})}
(\tfrac{-1}{\tau})=\{-\infty\}\cup\{\tau^2E_j\,|\, ord_{E_j}Q^{(n_{0}, n_{1},
n_{2}, n_{3})}(\cdot;\tau)=1\}.
\end{equation}

Now we prove by induction that for any $1\leq k\leq 2g$, $E_{k-1}\neq E_{k}$.

Suppose $E_0=E_1$, then (\ref{Ejdouble}) says that $E_0\notin\partial\tilde{\mathcal{S}}^{(n_{0},n_{1},n_{2},n_{3})}
(\tau)$, namely there are $\tilde{E}>E_0$ and $\varepsilon>0$ such that $[E_0-\varepsilon, \tilde{E}]\subset \tilde{\mathcal{S}}^{(n_{0},n_{1},n_{2},n_{3})}
(\tau)$ with $\tilde{E}\in\partial\tilde{\mathcal{S}}^{(n_{0},n_{1},n_{2},n_{3})}
(\tau)$. Then (\ref{Sboundary}) implies $\tilde{E}\in \{E_j\}_{j=0}^{2g}$, a contradiction with $\tilde{E}>E_0=\max_jE_j$. This proves $E_0\neq E_1$.

Assume by induction that for any $1\le i\leq k$, where $1\leq k\leq 2g-1$, we have $E_{i-1}\neq E_{i}$, i.e.
\[E_k<E_{k-1}<\cdots<E_1<E_0.\] We need to prove $E_k> E_{k+1}$. Suppose by contradiction that $E_{k}=E_{k+1}$.

{\bf Case 1.} $k$ is even.

Then it follows from $\{E_j | j\le k-1\}\subset \partial\tilde{\mathcal{S}}^{(n_{0},n_{1},n_{2},n_{3})}
(\tau)$ and (\ref{S-tilde}) that
\[[E_{k-1},E_{k-2}]\cup\cdots \cup[E_1,E_0]\subset \tilde{\mathcal{S}}^{(n_{0},n_{1},n_{2},n_{3})}
(\tau)\]
and
\[E<E_{k-1},\;\forall\, E\in \tilde{\mathcal{S}}^{(n_{0},n_{1},n_{2},n_{3})}
(\tau)\setminus [E_{k-1},E_{k-2}]\cup\cdots \cup[E_1,E_0].\]
So $E_{k}=E_{k+1}$, (\ref{S-tilde}) and (\ref{Ejdouble}) imply that there are $E_{k}<\tilde{E}_k<E_{k-1}$ and $\varepsilon>0$ such that $[E_k-\varepsilon, \tilde{E}_k]\subset \tilde{\mathcal{S}}^{(n_{0},n_{1},n_{2},n_{3})}
(\tau)$ with $\tilde{E}_k\in\partial\tilde{\mathcal{S}}^{(n_{0},n_{1},n_{2},n_{3})}
(\tau)$. Again it follows from (\ref{Sboundary}) that $\tilde{E}_k\in \{E_j\}_{j=0}^{2g}$, a contradiction with $E_{k}<\tilde{E}_k<E_{k-1}$. Thus Case 1 is impossible.

{\bf Case 2.} $k$ is odd.

Then it follows from $\{-\infty\}\cup\{\tau^2E_j | j\le k-1\}\subset \partial\tilde{\mathcal{S}}^{(n_{0},n_{2},n_{1},n_{3})}
(\tfrac{-1}{\tau})$ and $\tilde{\mathcal{S}}^{(n_{0},n_{2},n_{1},n_{3})}
(\tfrac{-1}{\tau})\subset \mathbb{R}$ that
\[(-\infty,\tau^2E_{0}]\cup [\tau^2 E_1,\tau^2E_2]\cdots\cup [\tau^2E_{k-2},\tau^2 E_{k-1}]\subset \tilde{\mathcal{S}}^{(n_{0},n_{2},n_{1},n_{3})}
(\tfrac{-1}{\tau})\]
and
\[E>\tau^2E_{k-1},\;\forall\, E\in \tilde{\mathcal{S}}^{(n_{0},n_{2},n_{1},n_{3})}
(\tfrac{-1}{\tau})\setminus (-\infty,\tau^2E_{0}]\cup\cdots\cup [\tau^2E_{k-2},\tau^2 E_{k-1}].\]
So $E_{k}=E_{k+1}$ and (\ref{Ejdouble1}) imply that there are $E_{k}<\tilde{E}_k<E_{k-1}$ and $\varepsilon>0$ such that $[\tau^2\tilde{E}_k,\tau^2E_k+\varepsilon]\subset \tilde{\mathcal{S}}^{(n_{0},n_{2},n_{1},n_{3})}
(\tfrac{-1}{\tau})$ with $\tau^2\tilde{E}_k\in\partial\tilde{\mathcal{S}}^{(n_{0},n_{2},n_{1},n_{3})}
(\tfrac{-1}{\tau})$. But then (\ref{Sboundary1}) implies $\tilde{E}_k\in \{E_j\}_{j=0}^{2g}$, a contradiction with $E_{k}<\tilde{E}_k<E_{k-1}$. Thus Case 2 is impossible.

This proves $E_{k}> E_{k+1}$. By induction we obtain $E_i\neq E_j$ for $i\neq j$. The proof is complete.
\end{proof}

Now we can give the proof of Theorem \ref{sharp-realroot}.

\begin{proof}[Proof of Theorem \ref{sharp-realroot}] Let $\tau\in i\mathbb{R}_{>0}$. If $\mathbf{n}$ satisfies neither (\ref{c1}) nor (\ref{c2}), then Theorem A says that all the roots of $Q^{\mathbf{n}}(E;\tau)$ are real and distinct, and so the spectrum $\sigma(L)$ is given by (\ref{spectrum0-1}).

Now suppose $\mathbf{n}$ satisfies either (\ref{c1}) or (\ref{c2}). Recall (\ref{eq-22}) that for $\tau\in i\mathbb{R}_{>0}$, $Q^{\mathbf{n}}(E;\tau)\in \mathbb{R}[E]$, so all its complex roots appear in pairs in $\mathbb{C}\setminus\mathbb{R}$, i.e. if $E\in \mathbb{C}\setminus\mathbb{R}$ is a root, so does its conjugate $\overline{E}$.
Define
\[\Gamma:=\{\tau\in i\mathbb{R}_{>0}: \text{$Q^{\mathbf{n}}(\cdot;\tau)$ has at least two roots in $\mathbb{C}\setminus\mathbb{R}$}\}.\]
Then Theorem A and Proposition \ref{real-imply-distinct} imply that $\Gamma\neq \emptyset$. Clearly $\Gamma$ is open in $i\mathbb{R}_{>0}$.
Furthermore, if $\tau_m\in \Gamma$ such that $\tau_m\to\tau\in i\mathbb{R}_{>0}\setminus \Gamma$ as $m\to\infty$, then the roots of $Q^{\mathbf{n}}(\cdot;\tau)$ are all real and $Q^{\mathbf{n}}(\cdot;\tau)$ must have a multiple root (i.e. the limit of the complex roots $E_m, \overline{E}_m$ of $Q^{\mathbf{n}}(\cdot;\tau_m)$ is a multiple root of $Q^{\mathbf{n}}(\cdot;\tau)$), a contradiction with Proposition \ref{real-imply-distinct}. This proves that $\Gamma$ is also closed in $i\mathbb{R}_{>0}$ and so $\Gamma=i\mathbb{R}_{>0}$. This also implies that for any $\tau\in i\mathbb{R}_{>0}$, $\sigma(L)\not\subset \mathbb{R}$ because the zero set of  $Q^{\mathbf{n}}(\cdot;\tau)$ is a proper subset of $\sigma(L)$. Recalling the fact (a) recalled in the proof of Lemma \ref{lemma5}, $\sigma(L)$ is still symmetric with respect to $\mathbb{R}$. This completes the proof.
\end{proof}

\section{Location of (anti)periodic eigenvalues}
\label{sec-periodic}

This section is devoted to the proof of Theorem \ref{thm-TV-open-gap}.
For this purpose, we need to consider the trigonometric limit $\tau\to i\infty$. It is well known that
\[\wp(z;\tau)\to \tfrac{\pi^2}{(\sin\pi z)^2}-\tfrac{\pi^2}{3},\quad \wp(z+\tfrac{1}{2};\tau)\to \tfrac{\pi^2}{(\cos\pi z)^2}-\tfrac{\pi^2}{3},\]
\[\wp(z+\tfrac{\omega_k}{2};\tau)\to -\tfrac{\pi^2}{3}, \quad k=2,3,\]
uniformly on compact sets of $\mathbb{C}\setminus\frac{1}{2}\mathbb{Z}$ as $\tau\to i\infty$.
Define
\begin{equation}\label{eq-sym}
q_{T}^{\mathbf{n}}(z):=-n_0(n_0+1)\tfrac{\pi^2}{(\sin\pi z)^2}-n_1(n_1+1)\tfrac{\pi^2}{(\cos\pi z)^2}+C_T^{\mathbf{n}},
\end{equation}
where
\begin{equation}
C_T^{\mathbf{n}}:=\frac{\pi^2}{3}\sum_{k=0}^3n_k(n_k+1).
\end{equation}
Then the above argument shows that
\[q^{\mathbf{n}}(z;\tau)\to q_{T}^{\mathbf{n}}(z)\quad\text{as}\;\tau\to i\infty.\]

Fix any $z_0\in\mathbb{C}\setminus\mathbb{R}$. Then for $\tau\in i\mathbb{R}_{>0}$ with $\operatorname{Im}\tau>|z_0|$, both
\[q(x;\tau):=q^{\mathbf{n}}(x+z_0;\tau)\quad\text{and}\quad q_T(x):=q_{T}^{\mathbf{n}}(x+z_0)\]
are smooth on $\mathbb{R}$ with period $\Omega=1$. Recalling Section \ref{Floquet-theory}, we denote the Hill's discriminants of
\begin{equation}\label{eq-hill-ell}Ly(x)=y''(x)+q(x;\tau)y(x)=Ey(x)\end{equation}
and
\begin{equation}\label{eq-hill-tri}y''(x)+q_T(x)y(x)=Ey(x)\end{equation}
by $\Delta(E;\tau)$ and $\Delta_T(E)$ respectively. Now we apply the following key fact about $\Delta_T(E)$: Since $q_{T}^{\mathbf{n}}(z)$ can be generated from $C_T^{\mathbf{n}}$ by finite times of Darboux transformations (see \cite[Remark 2.7]{GUW}), it is known (see e.g. \cite[Remark 1.3]{CL-CMP}) that $\Delta_T(E)$ coincides with the Hill's discriminant of
$y''(x)+C_T^{\mathbf{n}}y(x)=Ey(x)$ with respect to the period $1$, i.e.
\[\Delta_T(E)=2\cos\sqrt{C_T^{\mathbf{n}}-E},\]
Consequently,
\begin{equation}\label{eq-4-3}
\Delta_T^{-1}(\pm 2)=\{C_T^{\mathbf{n}}-j^2\pi^2\,|\, j\in\mathbb{Z}_{\geq 0}\}.
\end{equation}

\begin{lemma}\label{lem-tri-limit}
Under the above notations, we have
\begin{equation}\label{eq-4-1}\lim_{\tau\to i\infty} \Delta(E;\tau)=\Delta_T(E)=2\cos\sqrt{C_T^{\mathbf{n}}-E}.\end{equation}
\end{lemma}

\begin{proof}
Let $c_{\tau}(x;E)$ and $s_{\tau}(x;E)$ (resp. $c_{T}(x;E)$ and $s_{T}(x;E)$) be the special fundamental
system of solutions of (\ref{eq-hill-ell}) (resp. (\ref{eq-hill-tri})) satisfying the initial values
\[
c(0;E)=s^{\prime}(0;E)=1,\ c^{\prime}(0;E)=s(0;E)=0,
\]
then we have
\[
\Delta(E;\tau)=c_{\tau}(1;E)+s_{\tau}^{\prime}(1;E),\quad \Delta_T(E)=c_{T}(1;E)+s_{T}^{\prime}(1;E).
\]
Together with $q(x;\tau)\to q_T(x)$ uniformly on compact set of $\mathbb{R}$ as $\tau\to i\infty$, we obtain (\ref{eq-4-1}).
\end{proof}

Recalling $n_0=\max_k n_k\geq n_1$, it is well known that $q_{T}^{\mathbf{n}}(z)$ in (\ref{eq-sym}) is also a solution of the stationary KdV hierarchy with its spectral polynomial $Q_T^{\mathbf{n}}(E)$ given by
\begin{align}\label{eq-4-211}Q_{T}^{\mathbf{n}}(E)=&(E-C_T^{\mathbf{n}})
\prod_{j=1}^{n_0-n_1}(E-C_T^{\mathbf{n}}+j^2\pi^2)^2\nonumber\\
&\cdot\prod_{j=n_0-n_1+1}^{n_0}(E-C_T^{\mathbf{n}}+(2j-n_0+n_1)^2\pi^2)^2,\end{align}
where we use notation $\prod_{j=n_0-n_1+1}^{n_0}\ast=1$ if $n_1=0$. See e.g. \cite[Proposition 3.6]{CL-CMP}.  Here we have

\begin{lemma}\label{lem-spe-poly}
Suppose the genus $g$ in (\ref{genus}) satisfies $g=n_0$, i.e. $\deg Q^{\mathbf{n}}(E;\tau)=\deg Q_{T}^{\mathbf{n}}(E)=2n_0+1$. Then
\begin{equation}\label{eq-4-21}\lim_{\tau\to i\infty}Q^{\mathbf{n}}(E;\tau)
=Q_{T}^{\mathbf{n}}(E).\end{equation}
\end{lemma}

\begin{proof} In \cite{Tak1,Tak5}
Takemura already developed an algorithm of computing $\lim_{\tau\to i\infty}Q^{\mathbf{n}}(E;\tau)$ by decomposing $Q^{\mathbf{n}}(E;\tau)=\prod_{k=0}^3 P_k^{\mathbf{n}}(E;\tau)$, where $P_k^{\mathbf{n}}(E;\tau)$ is either $1$ or the characteristic polynomial of some matrix for each $k$; see particularly \cite[Appendix B]{Tak5}. In particular, Takemura's result implies that $\lim_{\tau\to i\infty}Q^{\mathbf{n}}(E;\tau)$ exists and can be computed explicitly for any given $\mathbf{n}$. Thus (\ref{eq-4-21}) can be proved by applying Takemura's algorithm.

Here we note that (\ref{eq-4-21}) can be also proved via the theory of the stationary KdV hierarchy. Since $q^{\mathbf{n}}(z;\tau)\to q_{T}^{\mathbf{n}}(z)$ as solutions of the stationary KdV hierarchy, and under our assumption their genus is the same, namely $\deg Q^{\mathbf{n}}(E;\tau)=\deg Q_{T}^{\mathbf{n}}(E)$, then the theory of the stationary KdV hierarchy (cf. \cite{GH-Book}) also implies (\ref{eq-4-21}) provided $\lim_{\tau\to i\infty}Q^{\mathbf{n}}(E;\tau)$ exists. We sketch the proof here for the reader's convenience.

First we review the basic setting on the stationary KdV hierarchy following \cite[Chapter 1]{GH-Book}. Given a meromorphic function $q(z)$, we define $\{f_{\ell}(q)\}_{\ell\in\mathbb{N}\cup\{0\}}$ recursively by
\begin{equation}\label{f-recursive}
f_0=1,\quad f_{\ell}'=-\tfrac{1}{4}f_{\ell-1}^{(3)}+qf_{\ell-1}'
+\tfrac{1}{2}q'f_{\ell-1}, \quad\ell\in\mathbb{N}.
\end{equation}
Explicitly, one finds
\[f_0=1,\quad f_1=\tfrac{1}{2}q+c_1,\]
\[f_2=-\tfrac{1}{8}(q''-3q^2)+c_1\tfrac{1}{2}q+c_2,\quad\text{etc.}\]
Here $\{c_{\ell}\}_{\ell\in\mathbb{N}}\subset\mathbb{C}$ denote integration constants that naturally arise when solving (\ref{f-recursive}). Subsequently, is will be convenient also to introduce the corresponding homogeneous coefficients $\hat{f}_{\ell}$ denoted by the vanishing of the integration constants $c_k$ for all $k$:
\[
\hat{f}_0=f_0=1,\quad \hat{f}_{\ell}=f_{\ell}\big|_{c_k=0, k=1,\cdots,\ell}\,,\quad \ell\in\mathbb{N}.
\]
Hence,
\begin{equation}\label{hat-f-2}
f_{\ell}=\sum_{k=0}^{\ell}c_{\ell-k}\hat{f}_k,\quad \ell\in\mathbb{N}\cup\{0\},\quad \text{where } c_0=1,
\end{equation}
and
\[
\hat{f}_0=1,\quad \hat{f}_1=\tfrac{1}{2}q,\quad\hat{f}_2=-\tfrac{1}{8}(q''-3q^2),\]
\begin{align*}
\hat{f}_3=\tfrac{1}{32}(q^{(4)}-10qq''-5q'^2+10q^3),\quad\text{etc.}
\end{align*}
It is known (cf. \cite[Theorem D.1]{GH-Book}) that $\hat{f}_{\ell}$ also satisfies (\ref{f-recursive}) and
\begin{equation}\label{hat-f-expression-sl}\hat{f}_{\ell}(q)\in \mathbb{Q}[q,q',q'',\cdots,q^{(2\ell-2)}],
\quad\ell\in\mathbb{N}.\end{equation}

Now consider a second-order differential operator of Schr\"{o}dinger-type
$L=\frac{d^2}{dz^2}+q(z)$
and a $2g+1$-order differential operator
\begin{equation}\label{eee}P_{2g+1}=\sum_{j=0}^g\left(f_{j}\frac{d}{dz}
-\tfrac{1}{2}f_{j}'\right)L^{g-j},\quad g\in\mathbb{N}\cup\{0\}.\end{equation}
By the recursion (\ref{f-recursive}), a direct computation leads to ($[\cdot,\cdot]$ the commutator symbol)
\[[L, P_{2g+1}]=-2f_{g+1}',\quad g\in\mathbb{N}\cup\{0\}.\]
In particular, $(L, P_{2g+1})$ represents the celebrated {\it Lax pair} of the KdV hierarchy. Varying $g\in\mathbb{N}\cup\{0\}$, the stationary KdV hierarchy
is then defined in terms of the vanishing of the commutator of $L$ and $P_{2g+1}$ by
\[
\text{s-KdV}_{g}(q):=[L, P_{2g+1}]=-2f_{g+1}'=0,\quad g\in\mathbb{N}\cup\{0\}.
\]

Now for the DTV potential $q^{\mathbf{n}}(z;\tau)$, there are integration constants $\{c_{\ell}^{\mathbf{n}}(\tau)\}_{\ell=1}^{g}$ such that the corresponding $P_{2g+1}^{\mathbf{n}}(\tau)=P_{2g+1}$ given in (\ref{hat-f-2})-(\ref{eee}) satisfies
\[[\tfrac{d^2}{dz^2}+q^{\mathbf{n}}(z;\tau), P_{2g+1}^{\mathbf{n}}(\tau)]=0.\]
On the other hand, it is known (\cite[Appendix D]{GH-Book}) that each integration constant $c_{\ell}^{\mathbf{n}}(\tau)\in \mathbb{Q}[E_0(\tau),\cdots, E_{2g}(\tau)]$, where $E_0(\tau),\cdots, E_{2g}(\tau)$ denote all the roots of the spectral polynomial $Q^{\mathbf{n}}(E;\tau)$. Since $\lim_{\tau\to i\infty}Q^{\mathbf{n}}(E;\tau)$ exists, we see that $c_{\ell}^{\mathbf{n}}(\tau)$ converges. From here, $q^{\mathbf{n}}(z;\tau)\to q_{T}^{\mathbf{n}}(z)$ and (\ref{hat-f-2})-(\ref{eee}), we conclude that
\[P_{2g+1}^{\mathbf{n}}:=\lim_{\tau\to i\infty}P_{2g+1}^{\mathbf{n}}(\tau)=\left(\tfrac{d}{dz}\right)^{2g+1}+\cdots\]
is a well-defined differential operator of order $2g+1=2n_0+1$ and
\[[\tfrac{d^2}{dz^2}+q_T^{\mathbf{n}}(z), P_{2g+1}^{\mathbf{n}}]=0.\]
Then, as recalled in (\ref{eqeq}), we obtain the following relations
\[P_{2g+1}^{\mathbf{n}}(\tau)^2=Q^{\mathbf{n}}(\tfrac{d^{2}}{dz^{2}}
+q^{\mathbf{n}}(z;\tau);\tau),\]
\[(P_{2g+1}^{\mathbf{n}})^2=Q_T^{\mathbf{n}}(\tfrac{d^{2}}{dz^{2}}
+q_T^{\mathbf{n}}(z)),\]
and so (\ref{eq-4-21}) holds.
\end{proof}

\begin{remark}\label{rmk1}
Given $\mathbf{n}=(n_0,n_1,n_2,n_3)$ with $n_k\in\mathbb{Z}_{\geq 0}$ and $n_0=\max n_k\geq 1$, we assume that $\sum n_k$ is odd and define $\tilde{\mathbf{n}}=(l_0,l_1,l_2,l_3)$ by
{\allowdisplaybreaks
\begin{align*}
& l_{0}= (n_{0}+n_{1}+n_{2}+n_{3}+1)/2\\
& l_{1}=\max\{\tilde{l}_1, -\tilde{l}_1-1\},\quad \tilde{l}_1:= (n_{0}+n_{1}-n_{2}-n_{3}-1)/2\\
& l_{2}=\max\{\tilde{l}_2, -\tilde{l}_2-1\},\quad \tilde{l}_2:= (n_{0}-n_{1}+n_{2}-n_{3}-1)/2\\
& l_{3}=\max\{\tilde{l}_3, -\tilde{l}_3-1\},\quad \tilde{l}_3:= (n_{0}-n_{1}-n_{2}+n_{3}-1)/2.
\end{align*}
}%
Then it was proved by Takemura \cite[Section 4]{Tak5} that
$y''(z)=[-q^{\mathbf{n}}(z;\tau)+E]y(z)$
and $y''(z)=[-q^{\tilde{\mathbf{n}}}(z;\tau)+E]y(z)$ are isomonodromic (i.e. their monodromy representations are the same) for any $(E,\tau)$, which immediately implies $Q^{\mathbf{n}}(E;\tau)=Q^{\tilde{\mathbf{n}}}(E;\tau)$. Here together with Lemma \ref{lem-spect-in}, we see that the spectrum $\sigma(\tilde{L})$ of
$\tilde{L}=\frac{d^2}{dx^2}+q^{\tilde{\mathbf{n}}}(x+z_0;\tau)$ is the same as $\sigma(L)$ of $L=\frac{d^2}{dx^2}+q^{\mathbf{n}}(x+z_0;\tau)$.
\end{remark}

\begin{remark} From the physical motivation, Takemura \cite{Tak2} studied the holomorphic dependence of certain $L^2$-integrable eigenvalues on $p=e^{\pi i \tau}$ as power series of $p$ as $\tau\to i\infty$; see \cite{Tak2} precise statements. In this paper, though we do not need to use the holomorphic dependence of (anti)periodic eigenvalues on $p=e^{\pi i \tau}$, but some idea of \cite{Tak2} was developed further in \cite{Tak5} and plays an important role in our proof of Lemma \ref{lem-spe-poly} and so in Theorem \ref{thm-TV-open-gap}.
\end{remark}

Now we are in the position to prove Theorem \ref{thm-TV-open-gap}.

\begin{proof}[Proof of Theorem \ref{thm-TV-open-gap}] Let $\mathbf{n}$ satisfy neither (\ref{c1}) nor (\ref{c2}), namely one of Cases (a)-(c) holds. Since $\tau=ib$ with $b>0$, it follows from Theorem A and (\ref{spe-S}) that
\[Q^{\mathbf{n}}(E;\tau)=\prod_{j=0}^{2g}(E-E_{j}(\tau))\]
with $E_{2g}(\tau)<\cdots<E_{0}(\tau)$, and the spectrum $\sigma(L_{\tau}):=\sigma(L)$ of $L_{\tau}:=L=\frac{d^2}{dx^2}+q(x;\tau)$ is given by
\begin{align}\label{eq-4-7}\sigma(L_{\tau})&=\{E\in\mathbb{C}\,|\, -2\leq \Delta(E;\tau)\leq 2\}\nonumber\\
&=(-\infty, E_{2g}(\tau)]\cup[E_{2g-1}(\tau), E_{2g-2}(\tau)]\cup \cdots \cup[E_{1}(\tau), E_{0}(\tau)].\end{align}
Since $E$ is a (anti)periodic eigenvalue of $L(\tau)$ if and only if $\Delta(E;\tau)=\pm 2$, so
\begin{equation}\label{eq-4-2}
\Delta(E_{j}(\tau);\tau)=\pm 2,\quad \forall j,
\end{equation}
\[\sigma_{p}(L_{\tau})=\{E\in\mathbb{C}\,|\,\Delta(E;\tau)=\pm 2\}\setminus\{E_{j}(\tau), j\in [0,2g]\}.\]

Recalling that $\Delta(E;\tau)$ is holomorphic in $E$, so for any $1\leq j\leq g$, if $\tilde{E}\in (E_{2j-1}(\tau),E_{2j-2}(\tau))$ is a local minimum point (resp. a local maximum point) of $\Delta(\cdot;\tau)$ on $(E_{2j-1}(\tau),E_{2j-2}(\tau))$, then
\begin{equation}\label{eq-4-5}
\Delta(\tilde{E};\tau)=-2\quad(\text{resp.}\quad \Delta(\tilde{E};\tau)=2).
\end{equation}
Indeed, if $\tilde{E}$ is a local minimum point of $\Delta(\cdot;\tau)$ on $(E_{2j-1}(\tau),E_{2j-2}(\tau))$ and $\Delta(\tilde{E};\tau)\in (-2, 2)$, then $\frac{d}{dE}\Delta(\tilde{E};\tau)=0$ and so it follows from
\begin{equation}\label{eq-Taylor}\Delta(E;\tau)-\Delta(\tilde{E};\tau)=a(E-\tilde{E})^k+o((E-\tilde{E})^k),\; a\neq 0, \; k\geq 2\end{equation}
and $\sigma(L_{\tau})=\{E\in\mathbb{C}\,|\, -2\leq \Delta(E;\tau)\leq 2\}$ that there are $2k\geq 4$ semi-arcs of $\sigma(L_{\tau})$ meeting at $\tilde{E}$, a contradiction with (\ref{eq-4-7}).

{\bf Step 1.} We consider Case (a).

Then it follows from (\ref{genus1})-(\ref{genus2}) that $g=n_0$ and $m=n_0-n_1$.
Therefore, Lemma \ref{lem-spe-poly} applies and we conclude from (\ref{eq-4-211})-(\ref{eq-4-21}) that
\begin{equation}\label{lim1}\lim_{\tau\to i\infty}E_0(\tau)=C_T^{\mathbf{n}},\end{equation}
\[\lim_{\tau\to i\infty}E_{2j-1}(\tau)=\lim_{\tau\to i\infty}E_{2j}(\tau)=C_T^{\mathbf{n}}-j^2\pi^2,\quad 1\leq j\leq m,\]
\begin{equation}\label{lim2}\lim_{\tau\to i\infty}E_{2j-1}(\tau)=\lim_{\tau\to i\infty}E_{2j}(\tau)=C_T^{\mathbf{n}}-(2j-m)^2\pi^2,\; m< j\leq g.\end{equation}

{\bf Case 1.} $1\leq j\leq m$.
Note that if $m=0$, then this case does not happen. So we assume $m\geq 1$.

Since \begin{equation}\label{eq-4-6}[E_{2j-1}(\tau),E_{2j-2}(\tau)]\to [C_T^{\mathbf{n}}-j^2\pi^2, C_T^{\mathbf{n}}-(j-1)^2\pi^2]\end{equation} as $\tau\to i\infty$, we conclude from (\ref{eq-4-1}) and (\ref{eq-4-2}) that
\[\Delta(E_{2j-1}(\tau);\tau)=(-1)^j2,\quad \Delta(E_{2j-2}(\tau);\tau)=(-1)^{j-1}2,\]
hold for all $\tau\in i\mathbb{R}_{>0}$ via the continuity of $\Delta(E,\tau)$ with respect to $(E,\tau)$.

Now we claim that for any $\tau\in i\mathbb{R}_{>0}$,
\begin{equation}\label{eq-4-4}
\sigma_{p}(L_{\tau})\cap(E_{2j-1}(\tau), E_{2j-2}(\tau))=\emptyset,
\end{equation}
namely
\[\Delta((E_{2j-1}(\tau), E_{2j-2}(\tau));\tau)=(-2,2).\]

Without loss of generality, we may assume that $j$ is odd (the case that $j$ is even can be proved in the same way).
First we show that (\ref{eq-4-4}) holds for $b=\operatorname{Im}\tau$ large. If not, there exists $\tau_k=ib_k$ with $b_k\to +\infty$ such that
\[\sigma_{p}(L_{\tau_k})\cap(E_{2j-1}(\tau_k), E_{2j-2}(\tau_k))\neq\emptyset.\]
This together with (\ref{eq-4-5}) imply the existence of $E_{1,k}, E_{2,k}\in \sigma_{p}(L_{\tau_k})$ satisfying
\[E_{2j-1}(\tau_k)<E_{1,k}<E_{2,k}< E_{2j-2}(\tau_k),\]
\[\Delta(E_{2j-1}(\tau_k);\tau_k)=\Delta(E_{2,k};\tau_k)=-2,\]
\[\Delta(E_{2j-2}(\tau_k);\tau_k)=\Delta(E_{1,k};\tau_k)=2.\]
By (\ref{eq-4-3}), (\ref{eq-4-1}) and (\ref{eq-4-6}), we obtain
\begin{equation}\label{eq-4-9}C_T^{\mathbf{n}}-(j-1)^2\pi^2=\lim_{k\to\infty}E_{1,k}\leq \lim_{k\to\infty}E_{2,k}=C_T^{\mathbf{n}}-j^2\pi^2,\end{equation}
clearly a contradiction.

Therefore, (\ref{eq-4-4}) holds for $b$ large. Define
\[\tilde{b}:=\inf\{b_0>0\,|\,\text{(\ref{eq-4-4}) holds for all $b>b_0$}\}\]
and suppose $\tilde{b}>0$. Then (\ref{eq-4-4}) holds for all $b>\tilde{b}$. If (\ref{eq-4-4}) holds for $b=\tilde{b}$, then the definition of $\tilde{b}$ implies the existence of $b_k\uparrow \tilde{b}$ such that (\ref{eq-4-4}) does not holds for $\tau_k=ib_k$, so the same argument as (\ref{eq-4-9}) shows $E_{2j-2}(i\tilde{b})\le E_{2j-1}(i \tilde{b})$, a contradiction. Hence (\ref{eq-4-4}) does not hold
for $\tilde{\tau}=i\tilde{b}$. Again this implies the existence of $\tilde{E}_{1}, \tilde{E}_{2}$ satisfying
\[E_{2j-1}(\tilde{\tau})<\tilde{E}_{1}<\tilde{E}_{2}< E_{2j-2}(\tilde{\tau}),\]
\[\Delta(E_{2j-1}(\tilde{\tau});\tilde{\tau})=\Delta(\tilde{E}_{2};\tilde{\tau})=-2
<\Delta(\tilde{E}_{1};\tilde{\tau})=2.\]
Then for $\tau=ib$ with $b-\tilde{b}>0$ sufficiently small, $\Delta(\cdot;\tau)$ has a local maximum point $E_{\tau}\in (E_{2j-1}(\tau), \tilde{E}_{2})$. However, (\ref{eq-4-5}) implies  $E_{\tau}\in \sigma_{p}(L_{\tau})\cap (E_{2j-1}(\tau), E_{2j-2}(\tau))$, a contradiction with the definition of $\tilde{b}$.

Therefore, $\tilde{b}=0$ and so (\ref{eq-4-4}) holds for all $\tau\in i\mathbb{R}_{>0}$.

{\bf Case 2.}  $m+1\leq j\leq g$.
Since \begin{align*}&[E_{2j-1}(\tau),E_{2j-2}(\tau)]\to \\ &[C_T^{\mathbf{n}}-(2j-m)^2\pi^2, C_T^{\mathbf{n}}-(2j-2-m)^2\pi^2]\end{align*} as $\tau\to i\infty$, we conclude from (\ref{eq-4-1}) and (\ref{eq-4-2}) that
\[\Delta(E_{2j-1}(\tau);\tau)=\Delta(E_{2j-2}(\tau);\tau)=(-1)^{m}2\]
hold for all $\tau\in i\mathbb{R}_{>0}$. Then by (\ref{eq-4-5}), there is a smallest \[\tilde{E}(\tau)\in \sigma_{p}(L_{\tau})\cap(E_{2j-1}(\tau),E_{2j-2}(\tau))\] satisfying
$\Delta(\tilde{E}(\tau);\tau)=(-1)^{m+1}2$.
From here and (\ref{eq-4-3})-(\ref{eq-4-1}), we obtain
\[\lim_{\tau\to i\infty}\tilde{E}(\tau)=C_T^{\mathbf{n}}-(2j-1-m)^2\pi^2.\]
Then the same argument as Case 1 shows that
\[\sigma_{p}(L_{\tau})\cap(E_{2j-1}(\tau),\tilde{E}(\tau))=\emptyset,\]
\[\sigma_{p}(L_{\tau})\cap(\tilde{E}(\tau),E_{2j-2}(\tau))=\emptyset.\]
In conclusion,
\[\sigma_{p}(L_{\tau})\cap(E_{2j-1}(\tau),E_{2j-2}(\tau))=\{\tilde{E}(\tau)\}.\]
This completes the proof for Case (a).

{\bf Step 2.} We consider Case (b): $n_0+n_3=n_1+n_2-1$.

Then (\ref{genus1})-(\ref{genus2}) says $g=n_0+n_3+1>n_0$ and $m=n_2+n_3+1$, so Lemma \ref{lem-spe-poly} does not apply. However, by Remark \ref{rmk1} we have $Q^{\mathbf{n}}(E;\tau)=Q^{\tilde{\mathbf{n}}}(E;\tau)$, where $\tilde{\mathbf{n}}=(l_0,l_1,l_2,l_3)$ with
\[l_{0}= (n_{0}+n_{1}+n_{2}+n_{3}+1)/2=n_0+n_3+1=g,\]
\[l_{1}=\tilde{l}_1= (n_{0}+n_{1}-n_{2}-n_{3}-1)/2=n_0-n_2,\]
\[l_{2}=\tilde{l}_2= (n_{0}-n_{1}+n_{2}-n_{3}-1)/2=n_0-n_1,\]
\[l_{3}= -\tilde{l}_3-1=-1-(n_{0}-n_{1}-n_{2}+n_{3}-1)/2=0.\]
Clearly (\ref{genus}) says that $\deg Q^{\tilde{\mathbf{n}}}(E;\tau)=2g+1=2l_0+1=\deg Q_{T}^{\tilde{\mathbf{n}}}(E)$ and $m=l_0-l_1$, so Lemma \ref{lem-spe-poly} implies
\begin{equation}\label{eq-lll}\lim_{\tau\to i\infty}Q^{\mathbf{n}}(E;\tau)=\lim_{\tau\to i\infty}Q^{\tilde{\mathbf{n}}}(E;\tau)
=Q_{T}^{\tilde{\mathbf{n}}}(E),\end{equation}
namely (\ref{lim1})-(\ref{lim2}) with $C^{\mathbf{n}}_T$ replaced by $C^{\tilde{\mathbf{n}}}_T$ hold. Then the same proof as Step 1 yields the desired assertions.

Here we emphasize that $Q_{T}^{\mathbf{n}}(E)\neq Q_{T}^{\tilde{\mathbf{n}}}(E)$ and $\lim_{\tau\to i\infty}Q^{\mathbf{n}}(E;\tau)\neq Q_{T}^{\mathbf{n}}(E)$ because their degrees are not the same.

{\bf Step 3.} We consider Case (c): $n_0+n_3=n_1+n_2+1$ and $n_3\geq 1$.

Then (\ref{genus1})-(\ref{genus2}) says $g=n_0+n_3>n_0$ and
 \[m=%
\begin{cases}
n_2+n_3+1 & \text{if $n_0>n_2$},\\
n_2+n_3   &\text{if $n_0=n_2$},
\end{cases}\]so Lemma \ref{lem-spe-poly} does not apply. Again by Remark \ref{rmk1} we have $Q^{\mathbf{n}}(E;\tau)=Q^{\tilde{\mathbf{n}}}(E;\tau)$, where $\tilde{\mathbf{n}}=(l_0,l_1,l_2,l_3)$ with
\[l_{0}= (n_{0}+n_{1}+n_{2}+n_{3}+1)/2=n_0+n_3=g,\]
\[\tilde{l}_1= (n_{0}+n_{1}-n_{2}-n_{3}-1)/2=n_0-n_2-1,\]
\[\text{i.e.}\quad l_1=\max\{\tilde{l}_1,-1-\tilde{l}_1\}=\begin{cases}
n_0-n_2-1 & \text{if $n_0>n_2$},\\
0  &\text{if $n_0=n_2$},
\end{cases}\]
\[\tilde{l}_2= (n_{0}-n_{1}+n_{2}-n_{3}-1)/2=n_0-n_1-1,\]
\[\text{i.e.}\quad l_2=\max\{\tilde{l}_2,-1-\tilde{l}_2\}=\begin{cases}
n_0-n_1-1 & \text{if $n_0>n_1$},\\
0  &\text{if $n_0=n_1$},
\end{cases}\]
and
\[l_{3}= \tilde{l}_3=(n_{0}-n_{1}-n_{2}+n_{3}-1)/2=0.\]
Again (\ref{genus}) says that $\deg Q^{\tilde{\mathbf{n}}}(E;\tau)=2g+1=2l_0+1=\deg Q_{T}^{\tilde{\mathbf{n}}}(E)$ and $m=l_0-l_1$, so Lemma \ref{lem-spe-poly} implies
(\ref{eq-lll}) and hence (\ref{lim1})-(\ref{lim2}) with $C^{\mathbf{n}}_T$ replaced by $C^{\tilde{\mathbf{n}}}_T$ hold. The rest proof is the same as Step 1.

The proof is complete.
\end{proof}

\medskip

\noindent{\bf Acknowledgements} The authors thank Professor Veselov for pointing out that the DTV potential was first introduced by Darboux \cite{Darboux}. The research of Z. Chen was supported by NSFC (No. 11871123) and Tsinghua University Initiative Scientific Research Program (No. 2019Z07L02016).

\end{document}